\documentclass[11pt,reqno]{amsart}

 \usepackage{amsfonts,amsmath,amscd,amssymb,amsbsy,amsthm,amstext,amsopn,amsxtra}
 \usepackage{color,fullpage,mathrsfs}
 \usepackage{subfigure}
 \usepackage{verbatim} 
 \usepackage{stmaryrd}
 \usepackage{extarrows}
 \usepackage[all]{xy}
 \usepackage{bm}   
 \usepackage{hyperref}
 \usepackage{enumitem}
 
\usepackage{cancel}
\usepackage{color}
\usepackage{tikz}
 \usepackage{tikz-cd}

 \numberwithin{equation}{section}
 \hypersetup{colorlinks,linkcolor=[rgb]{0,0,1},citecolor=[rgb]{1,0,0}}

\newtheorem{theorem}{Theorem}[section]
\newtheorem{lemma}[theorem]{Lemma}
\newtheorem{proposition}[theorem]{Proposition}
\newtheorem{corollary}[theorem]{Corollary}

\theoremstyle{definition}

\newtheorem{example}[theorem]{Example}

\newtheorem{remark}[theorem]{Remark}

\newcommand{\one}{\ensuremath{(\mathrm{i})}}
\newcommand{\two}{\ensuremath{(\mathrm{ii})}}

\newcommand{\aLus}{\ensuremath{(\mathrm{a})}}
\newcommand{\bLus}{\ensuremath{(\mathrm{b})}}

\newcommand{\kk}{\ensuremath{\Bbbk}} 
\newcommand{\NN}{\ensuremath{\mathbb{N}}}

\newcommand{\git}{\ensuremath{\operatorname{\!/\!\!/\!}}}
\newcommand{\GL}{\operatorname{GL}}
\newcommand{\head}{\operatorname{h}}
\newcommand{\Hilb}{\operatorname{Hilb}}

\newcommand{\id}{\operatorname{id}}

\newcommand{\Mat}{\operatorname{Mat}}

\newcommand{\Rep}{\operatorname{Rep}}

\newcommand{\SL}{\operatorname{SL}}

\newcommand{\Spec}{\operatorname{Spec}}
\newcommand{\Sym}{\operatorname{Sym}}

\newcommand{\tail}{\operatorname{t}}
\newcommand{\tr}{\operatorname{tr}}

\newcommand{\xpij}[2]{\operatorname{x}_{#1,#2}}
\newcommand{\xaij}[2]{\operatorname{x}_{#1,#2}}

\newcommand{\pij}{\operatorname{x}_{p,ij}}

\newcommand{\aij}{\operatorname{x}_{a,ij}}

\newcommand{\plij}{\operatorname{x}_{p_l,ij}}

\title{The Le Bruyn--Procesi theorem following Lusztig}

\author{Alastair Craw} 
\author{Ryo Yamagishi} 

\address{Department of Mathematical Sciences, 
University of Bath, 
Claverton Down, 
Bath BA2 7AY, 
UK.}
\email{a.craw@bath.ac.uk}

\address{Department of Mathematics, Faculty of Science, Kyushu University, 744, Motooka, Nishi-ku, Fukuoka, Japan 819-0395.}
 \email{yamagishi@kyushu-u.ac.jp}

\begin{document}
 \begin{abstract}
 For any quiver $Q$ and dimension vector $v$, Le Bruyn--Procesi proved that the invariant ring for the action of the change of basis group on the space of representations $\Rep(Q,v)$ is generated by the traces of matrix products associated to cycles in the quiver. Lusztig generalised this to allow for vertices where the group acts trivially. Here we provide a simple new proof of Lusztig's theorem and determine the relations between his algebra generators for any quiver with relations. 
\end{abstract} 

 \maketitle

\tableofcontents

\section{Introduction}

  Let $Q$ be a finite, connected quiver with vertex set $Q_0$ and arrow set $Q_1$, let $v=(v_i)\in \mathbb{N}^{Q_0}$ be a dimension vector, and let $\kk$ be an algebraically closed field of characteristic zero.  The representation space for the pair $(Q,v)$ is the $\kk$-vector space
\begin{equation}
    \label{eqn:Repintro}
\Rep(Q,v)=\bigoplus_{a\in Q_1} \Mat\big(v_{\head(a)}\times v_{\tail(a)}\big)
\end{equation}
of tuples of matrices with entries in $\kk$,  where $\head(a)$ and $\tail(a)$ denote the vertices at the head and tail of $a$ respectively.  The group $\GL(v):=\prod_{i\in Q_0} \GL(v_i,\kk)$
 acts by conjugation on $\Rep(Q,v)$, so it acts dually on the coordinate ring $\kk[\Rep(Q,v)]$. Each nontrivial cycle $\gamma$ in $Q$ defines a trace function $\tr_\gamma$ that associates to each representation of $Q$, the trace of the product of matrices determined by the arrows in the cycle $\gamma$. A well-known theorem of Le Bruyn--Procesi~\cite{LBP90} established that the trace functions provide $\kk$-algebra generators for the $\GL(v)$-invariant subalgebra $\kk[\Rep(Q,v)]^{\GL(v)}$.

 In practise, one typically studies representations of a quiver satisfying relations, or equivalently, modules over an algebra $A$ obtained as a quotient of the path algebra of $Q$. In this context, by restricting functions to the representation subscheme $\Rep(A,v)\subseteq \Rep(Q,v)$ cut out by the ideal of relations, we obtain a surjective $\kk$-algebra homomorphism
 \begin{equation}
    \label{eqn:tauintro}
 \tau\colon \kk[\Rep(Q,v)]^{\GL(v)}\longrightarrow \kk[\Rep(A,v)]^{\GL(v)},
 \end{equation}
 so the trace functions also generate the $\kk$-algebra $\kk[\Rep(A,v)]^{\GL(v)}$. Understanding this ring is of interest because it defines the affine GIT quotient $\Rep(A,v)\git \GL(v)$. 
 
 We extend the result of Le Bruyn--Procesi by calculating the kernel of the map from \eqref{eqn:tauintro} in terms of the quiver presentation $\beta\colon \kk Q\to A$, thereby describing the relations between the algebra generators of $\kk[\Rep(A,v)]^{\GL(v)}$ in terms of trace functions. To state the result, write $e_k\in A$ for the vertex idempotent obtained as the class of the trivial path at vertex $k\in Q_0$.

 \begin{theorem}
 \label{thm:LBPintro} 
 The algebra $\kk[\Rep(A,v)]^{\GL(v)}$ is generated by the trace functions $\tr_{\beta(\gamma)}$ associated to cycles $\gamma$ in $Q$, and moreover, the kernel of $\tau$ from \eqref{eqn:tauintro} is generated by the functions $\tr_{\gamma}$, where $\gamma\in e_k\ker(\beta) e_k$ for some $k\in Q_0$.  
 \end{theorem}
 
 In fact, we prove a more general result for a class of invariant subalgebras of $\kk[\Rep(Q,v)]$ studied first by Lusztig~\cite{Lusztig98}. To define the class,  consider for any subset $K\subseteq Q_0$, the action of the subgroup 
 \[
 G_K:= \prod_{k\in K} \GL(v_k,\kk)
 \]
 of $\GL(v)$.  For any nontrivial path $p$ in $Q$ and for indices $1\leq i\leq v_{\head(p)}$ and $1\leq j\leq v_{\tail(p)}$, there is a contraction function $\pij\in \kk[\Rep(Q,v)]$ that associates to each representation, the $(i,j)$-entry of the product of matrices determined by the arrows in the path $p$. Lusztig~\cite[Theorem~1.3]{Lusztig98} proved that the $G_K$-invariant subalgebra $\kk[\Rep(Q,v)]^{G_K}$ is generated by the trace functions $\tr_\gamma$ for cycles $\gamma$ in $Q$ that traverse only arrows in $Q$ with head and tail in $K$, together with the contraction functions $\pij$ for paths $p$ in $Q$ with head and tail in  $Q_0\setminus K$, where 
 $1\leq i\leq v_{\head(p)}$ and $1\leq j\leq v_{\tail(p)}$.
 
 For any quiver algebra with presentation $\beta\colon \kk Q\to A$ as above, restricting functions to the representation subscheme $\Rep(A,v)\subseteq \Rep(Q,v)$ defines a surjective $\kk$-algebra homomorphism
 \begin{equation}
    \label{eqn:tauKintro}
 \tau_K\colon \kk[\Rep(Q,v)]^{G_K}\longrightarrow \kk[\Rep(A,v)]^{G_K},
 \end{equation}
 so the corresponding trace and contraction functions also generate $\kk[\Rep(A,v)]^{G_K}$. Our main result generalises Theorem~\ref{thm:LBPintro} by calculating the kernel of the map from \eqref{eqn:tauKintro} in terms of the quiver presentation $\beta\colon \kk Q\to A$, thereby describing the relations between Lusztig's algebra generators of $\kk[\Rep(A,v)]^{G_K}$ in terms of trace and contraction functions (see Proposition~\ref{prop:RepAinvariants} and Theorem~\ref{thm:I_K}):

\begin{theorem}
\label{thm:introLBPF}
 The algebra $\kk[\Rep(A,v)]^{G_K}$ is generated by the functions:
 \begin{enumerate}
    \item[\aLus] $\tr_\gamma$ for cycles $\gamma$ in $Q$ that traverse only arrows in $Q$ with head and tail in $K$; and
    \item[\bLus] $\pij$ for paths $p$ in $Q$ with head and tail in  $Q_0\setminus K$, where 
 $1\leq i\leq v_{\head(p)}$ and $1\leq j\leq v_{\tail(p)}$.
 \end{enumerate}  
 Moreover, $\ker(\tau_K)$ is generated by the functions $\tr_{\gamma}$, where $\gamma\in e_k\ker(\beta) e_k$ for some $k\in K$, and the functions $\xpij{p}{ij}$, where $p\in e_{\head}\ker(\beta) e_{\tail}$ for some $\head, \tail\in Q_0\setminus K$, where $1\leq i\leq v_{\head}$ and $1\leq j\leq v_{\tail}$.
\end{theorem}

This result interpolates naturally between the description of the invariant ring $\kk[\Rep(A,v)]^{\GL(v)}$ from Theorem~\ref{thm:LBPintro} in the case $K=Q_0$, and the natural description of $\kk[\Rep(A,v)]$ as the quotient of the polynomial ring $\kk[\xpij{a}{ij} \mid a\in Q_1]$ by the ideal $\langle \xpij{p}{ij} \mid p\in \ker(\beta)\rangle$ in the case $K=\varnothing$. 

 \medskip

In Section~\ref{sec:LBPF}, we recall Lusztig's generalisation of the Le Bruyn--Procesi theorem, and we provide a new proof of this result that adapts the framing trick of Crawley-Boevey~\cite[Section~1]{CBmoment} to construct a quiver that we build from $(Q,v)$ and our choice of subset $K$. Applying the original Le Bruyn--Procesi theorem to this framed quiver reproduces precisely Lusztig's algebra generators of $\kk[\Rep(Q,v)]^{G_K}$, which in turn allows us to determine the algebra generators of $\kk[\Rep(A,v)]^{G_K}$ from Theorem~\ref{thm:introLBPF} in Section~\ref{sec:quiveralgebras}.
Our main new contribution comes in Section~\ref{sec:relations} where we compute the relations between Lusztig's generators. Here, once again, our proof introduces an auxiliary quiver obtained from $Q$; this time, we add extra arrows for a given choice of generators of the two-sided ideal $\ker(\beta)$ in $\kk Q$. We conclude by presenting an illustrative example in Section~\ref{sec:example}.

\medskip

\noindent \textbf{Notation.} 
Throughout the paper, $\kk$ is an algebraically closed field of characteristic zero. 
Let $\Mat(m\times n)$ be the $\kk$-vector space of $m\times n$ matrices with entries in $\kk$. We write nontrivial paths in $Q$ as products $p = a_\ell\cdots a_1$ of arrows from right to left, starting at vertex $\tail(p):=\tail(a_1)$ and ending at vertex $\head(p):=\head(a_{\ell})$. A cycle in $Q$ is a path $p$ such that $\tail(p)=\head(p)$.

\medskip

\noindent \textbf{Acknowledgements.} We thank an anonymous referee for explaining that Theorem~\ref{thm:Lusztig} was proved first by Lusztig. Both authors were supported by Research Project Grant RPG-2021-149 from The Leverhulme Trust. The second author was also supported by JSPS KAKENHI Grant JP19K14504.

  \section{The Le Bruyn--Procesi theorem with frozen vertices}
  \label{sec:LBPF}
  
  Let $Q$ be a finite, connected quiver with vertex set $Q_0$ and arrow set $Q_1$. 
  Let $v=(v_i)\in \mathbb{N}^{Q_0}$ be a dimension vector. The representation space for the pair $(Q,v)$ is the $\kk$-vector space
\begin{equation}
    \label{eqn:Repold}
\Rep(Q,v)=\bigoplus_{a\in Q_1} \Mat(v_{\head(a)}\times v_{\tail(a)}),
\end{equation}
 where $\head(a)$ and $\tail(a)$ denote the vertices at the head and tail respectively of $a$. 
 
 For any nontrivial cycle $\gamma=a_\ell\cdots a_1$ in $Q$ obtained by traversing the arrows $a_1, a_2, \dots a_\ell$ in this order, the map sending $B=(B_a)_{a\in Q_1}\in \Rep(Q,v)$ to
 \[
 \tr_\gamma(B):= \tr\big(B_{a_\ell}\cdots B_{a_1}\big)
 \]
 is a $\GL(v)$-invariant polynomial function $\tr_\gamma\in \kk[\Rep(Q,v)]$ called the \emph{trace function} for the cycle $\gamma$. For each $i\in Q_0$, define the trace of the trivial path $e_i$ at vertex $i$ to be the constant function with value $\tr_{e_i}:=v_i$. Let $\kk Q$ denote the path algebra of $Q$. For any vertex $i\in Q_0$, each element $\gamma \in e_i \kk Q e_i$ is a linear combination $\gamma=\sum_k \lambda_k \gamma_k$ of cycles passing through $i$, and we obtain a trace function for $\gamma$ by extending the above definition linearly over $\kk$, that is,  $\tr_\gamma:= \sum_k \lambda_k \tr_{\gamma_k}$.
 
 Let $p=a_\ell\cdots a_1$ be a nontrivial path in $Q$ obtained by traversing the arrows $a_1, a_2, \dots a_\ell$ from vertex $\tail(p)$ to vertex $\head(p)$. For $B=(B_a)\in \Rep(Q,v)$, we have $B_{a_\ell}\cdots B_{a_1}\in \Mat(v_{\head(p)}\times v_{\tail(p)})$. 
 For $1\leq i\leq v_{\head(p)}$ and $1\leq j\leq v_{\tail(p)}$, the map sending $B$ to the $(i,j)$-entry of this matrix, denoted
 \[
 \pij(B):= \big(B_{a_\ell}\cdots B_{a_1}\big)_{ij},
 \]
 defines a polynomial function $\pij\in \kk[\Rep(Q,v)]$;  following \cite{KraftProcesi}, we call this the \emph{$(i,j)$-contraction function} of the path $p$. More generally, for any vertices $\tail, \head\in Q_0$, each element $g\in e_{\head}\kk Q e_{\tail}$ is a linear combination $g=\sum \lambda_l p_l$ of paths each with head and tail at $\head$ and $\tail$ respectively, and we extend linearly to obtain polynomial function $\xpij{g}{ij} := \sum_l \lambda_l \xpij{p_l}{ij}$ for each $1\leq i\leq v_{\head}$ and $1\leq j\leq v_{\tail}$.
 
 The decomposition \eqref{eqn:Repold} shows that the coordinate ring of $\Rep(Q,v)$ is the polynomial ring
 \begin{equation}
     \label{eqn:kRep}
 \kk\big[\!\Rep(Q,v)\big] = \bigotimes_{a\in Q_1} \kk\big[\aij\mid 1\leq i\leq v_{\head(a)}\text{ and }1\leq j\leq v_{\tail(a)}\big].
 \end{equation}
 Each contraction function $\pij$ can be expressed as a sum of monomials in this ring as follows.
 
 \begin{lemma}
 \label{lem:product}
 Let $p, q$ be nontrivial paths in $Q$ with $\head(p)=\tail(q)$. Then $\xpij{qp}{ij} = \sum_{1\leq k\leq v_{\head(p)}} \xpij{q}{ik} \xpij{p}{kj}$. In particular, if $p=a_\ell\cdots a_1$ is a nontrivial path in $Q$, then 
 \[
 \pij = \sum_{k_1, \dots, k_{\ell-1}} \xpij{a_{\ell}}{ik_{\ell-1}}\xpij{a_{\ell-1}}{k_{\ell-1}k_{\ell-2}}\cdots \xpij{a_{1}}{k_{1}j}
 \]
 where the sum runs over all indices satisfying $1\leq k_i\leq v_{\head(a_i)}$ for $1\leq i\leq \ell$. 
 \end{lemma}
 \begin{proof}
    Write $p=a_\ell\cdots a_1$ and  $q=a_m\cdots a_{\ell+1}$. For $B\in \Rep(Q,v)$, the product of matrices gives
    \[
    \xpij{qp}{ij}(B)= \big(B_{a_m}\cdots B_{a_1}\big)_{ij} = \sum_{k=1}^{v_{\head(p)}} \big( B_{a_m}\cdots B_{a_{\ell+1}} \big)_{ik} \big(B_{a_\ell}\cdots B_{a_{1}} \big)_{kj} 
    = \sum_{k=1}^{v_{\head(p)}} \xpij{q}{ik}(B) \xpij{p}{kj}(B)   
    \]
    as required. The second statement follows by induction on the length of the path.
 \end{proof}

 The reductive group $\GL(v):=\prod_{i\in Q_0} \GL(v_i,\kk)$ acts by conjugation on $\Rep(Q,v)$, so it acts dually on the ring $\kk[\Rep(Q,v)]$. The diagonal scalar subgroup $\kk^\times = \{(\lambda\id_{v_i})_{i\in Q_0} \mid \lambda\neq 0\}$ of $\GL(v)$ acts trivially, so a polynomial is $\GL(v)$-invariant if and only if it is invariant under the action of $\GL(v)/\kk^\times$. A well-known result of Le Bruyn--Procesi~\cite{LBP90} established that the $\GL(v)$-invariant subalgebra $\kk[\Rep(Q,v)]^{\GL(v)}$ is generated by the trace functions $\tr_\Gamma$ associated to cycles $\gamma$ in $Q$.

 More generally, for any subset $K\subseteq Q_0$, consider the subgroup 
 \[
 G_K:= G_K(v)= \prod_{k\in K} \GL(v_k,\kk)
 \]
 of $\GL(v)$ indexed by vertices in $K$. 
 If the head and tail of a path $p$ lie in $K^c:=Q_0\smallsetminus K$, then each function $\pij$ is $G_K$-invariant. In this context, Lusztig~\cite[Theorem~1.3]{Lusztig98} established the following generalisation of the Le Bruyn--Procesi theorem.
   
\begin{theorem}[Lusztig]
\label{thm:Lusztig}
 For any subset $K\subseteq Q_0$, the algebra $\kk[\Rep(Q,v)]^{G_K}$ is generated by:
 \begin{enumerate}
 \item[\aLus] the trace functions $\tr_\gamma$ associated to cycles $\gamma$ in $Q$ that traverse only arrows in $Q$ with head and tail in $K$; and
 \item[\bLus] the contraction functions $\pij$ for paths $p$ in $Q$ with head and tail in $K^c$, where the indices satisfy $1\leq i\leq v_{\head(p)}$ and $1\leq j\leq v_{\tail(p)}$.
 \end{enumerate}
 \end{theorem}

This result interpolates naturally between the description of the invariant ring $\kk[\Rep(Q,v)]^{\GL(v)}$ by Le Bruyn--Procesi~\cite{LBP90} in the case $K=Q_0$, and the description of the polynomial ring $\kk[\Rep(Q,v)]$ from \eqref{eqn:kRep} in the case $K=\varnothing$. 

\begin{remark}
The invariant algebra $\kk[\Rep(Q,v)]^{G_K}$ contains the polynomial ring
\begin{equation}
    \label{eqn:kReppoly}
 \bigotimes_{\{a\in Q_1 \;\mid \;\head(a), \tail(a)\in K^c\}} \kk\big[\aij\mid 1\leq i\leq v_{\head(a)}\text{ and }1\leq j\leq v_{\tail(a)}\big].
 \end{equation}
\end{remark}

We present a simple new proof of Theorem~\ref{thm:Lusztig} that uses a framing trick similar to that introduced by Crawley-Boevey~\cite[Section~1]{CBmoment}. For any subset $K\subseteq Q_0$, we partition the arrow set $Q_1$ as the disjoint union of the following sets:
\begin{eqnarray*}
S_{0,0}& := & \{a\in Q_1 \mid \tail(a), \head(a)\in K^c\}; \\
S_{0,1} & := & \{a\in Q_1 \mid \tail(a)\in K^c, \head(a)\in K\}; \\
S_{1,0}& := & \{a\in Q_1 \mid \tail(a)\in K, \head(a)\in K^c\}; \\
S_{1,1}& := & \{a\in Q_1 \mid \tail(a), \head(a)\in K\}.
\end{eqnarray*}
Given a triple $(Q, v, K)$ as above, we construct a new quiver $Q'$ whose vertex set $Q_0^\prime$ is obtained from $K$ by adjoining a new vertex $\infty$, and whose arrow set $Q_1^\prime$ is obtained from  $S_{1,1}$ by adjoining two additional collections of arrows, namely:
\begin{enumerate}
    \item for each $b\in S_{0,1}$, add arrows $\alpha_b^1, \dots, \alpha_b^{v_{\tail(b)}}$ to $Q_1^\prime$, 
    each with tail at $\infty$ and head at $\head(b)$;
    \item for each $c\in S_{1,0}$, add arrows $\beta_c^1, \dots, \beta_c^{v_{\head(c)}}$ to $Q_1^\prime$, each with tail at $\tail(c)$ and head at $\infty$.
\end{enumerate}
 Now define a dimension vector $v^\prime\in \NN^{Q_0^\prime}=\NN\times \NN^{Q_0}$ by setting $v^\prime_\infty = 1$ and $v_k^\prime = v_k$ for $k\in K$. The representation space for the pair $(Q^\prime,v^\prime)$ is the $\kk$-vector space 
\begin{equation}
    \label{eqn:Repprime}
\Rep(Q^\prime,v^\prime)=\bigoplus_{a\in S_{1,1}} \Mat(v_{\head(a)}\times v_{\tail(a)}) \oplus \bigoplus_{b\in S_{0,1}} \Mat(v_{\head(b)}\times 1)^{\oplus v_{\tail(b)}}\oplus  \bigoplus_{c\in S_{1,0}} \Mat(1\times v_{\tail(c)})^{\oplus v_{\head(c)}}.
\end{equation}

 The group $\GL(v^\prime)= \prod_{k\in Q_0^\prime} \GL(v_k^\prime,\kk)$ acts on $\Rep(Q^\prime, v^\prime)$ by conjugation, and the diagonal scalar subgroup $\kk^\times$ acts trivially. Since $v_\infty^\prime = 1$, projection away from the $\infty$ coordinate identifies the action of the quotient group $\GL(v')/\kk^\times$ on $\Rep(Q',v')$ with the natural action of the subgroup $G_K\cong \prod_{k\in K} \GL(v_k)$ of $\GL(v')$.

 \begin{lemma}
 \label{lem:varphi}
 There is an $G_K$-equivariant isomorphism of $\kk$-vector spaces 
 \[
 \varphi\colon \Rep(Q,v)\longrightarrow 
\Rep(Q',v')\times \bigoplus_{a\in S_{0,0}} \Mat(v_{\head(a)}\times v_{\tail(a)}).
\]
 \end{lemma}
 \begin{proof}
  Define $\varphi$ to be the direct sum of four maps, one for each component in the decomposition
  \begin{equation}
      \label{eqn:rep4sums}
  \Rep(Q,v)=\bigoplus_{0\leq \lambda, \mu\leq 1} \left(\bigoplus_{a\in S_{\lambda, \mu}} \Mat(v_{\head(a)}\times v_{\tail(a)})\right),
   \end{equation}
   as follows. The group $G_K$ acts trivially on the summand indexed by $\lambda=\mu=0$ that also appears in the codomain of $\varphi$, while $G_K$ acts naturally by conjugation on the summand indexed by $\lambda=\mu=1$ that appears in the decomposition \eqref{eqn:Repprime} of $\Rep(Q',v')$. In each case, define the corresponding summand of $\varphi$ to be the identity map. The remaining two summands of $\varphi$ are the isomorphisms
   \[
  \bigoplus_{a\in S_{0,1}} \Mat(v_{\head(a)}\times v_{\tail(a)})\longrightarrow 
     \bigoplus_{a\in S_{0,1}} \Mat(v_{\head(a)}\times 1)^{\oplus v_{\tail(a)}}
    \]
   and
   \[
   \bigoplus_{a\in S_{1,0}} \Mat(v_{\head(a)}\times v_{\tail(a)})\longrightarrow 
     \bigoplus_{a\in S_{1,0}} \Mat(1\times v_{\tail(a)})^{\oplus v_{\head(a)}}
   \]
   that identify each matrix with its column vectors and its row vectors respectively. Observe that the first of these maps is $G_K$-equivariant with respect to the natural action of $G_K$ on the left, while the second is $G_K$-equivariant with respect to the natural action of $G_K$ on the right.  Therefore $\varphi$ is the direct sum of four $\kk$-linear isomorphisms, each of which is $G_K$-equivariant. 
  \end{proof}

 Recall that $f\in \kk[\Rep(Q',v')]$ is $\GL(v')$-invariant if and only if it is $\GL(v')/\kk^\times$-invariant. Since the action of $G_K$ on $\Rep(Q',v')$ is identified with the natural action of $\GL(v')/\kk^\times$, we obtain:
 
 \begin{corollary}
Pullback along $\varphi$ induces an isomorphism of invariant subalgebras
 \[
 \varphi^*\colon \kk[\Rep(Q',v')]^{\GL(v')}\otimes_{\kk} \bigotimes_{a\in S_{0,0}} \kk\big[\!\Mat(v_{\head(a)}\times v_{\tail(a)})\big] \longrightarrow \kk[\Rep(Q,v)]^{G_K}.
 \]
 \end{corollary}

\begin{proof}[Proof of Theorem~\ref{thm:Lusztig}]
Carry the generators of the domain of $\varphi^*$ across the isomorphism $\varphi^*$. First,  $\varphi^*$ identifies $\bigoplus_{a\in S_{0,0}} \kk\big[\!\Mat(v_{\head(a)}\times v_{\tail(a)})\big]$ with the subalgebra \eqref{eqn:kReppoly} of $\kk[\Rep(Q,v)]^{G_K}$, so 
\begin{equation}
    \label{eqn:easygens} 
\big\{\aij\mid a\in S_{0,0}\text{ where } 1\leq i\leq v_{\head(a)}\text{ and }1\leq j\leq v_{\tail(a)}\big\}
\end{equation}
 provide some of the algebra generators of $\kk[\Rep(Q,v)]^{G_K}$. The remaining generators are the images under $\varphi^*$ of the generators of $\kk[\Rep(Q',v')]^{\GL(v')}$ which, by the main result of Le Bruyn--Procesi~\cite{LBP90}, are the trace functions $\tr_{\gamma'}$ associated to cycles $\gamma'$ in $Q'$. There are two cases:

 \medskip

 \noindent \textsc{Case 1:} \emph{the cycle $\gamma'$ in $Q'$ traverses only arrows in the set $S_{1,1}$}.

 \smallskip

 Any such cycle $\gamma'=a_\ell\cdots a_1$ may be viewed as a cycle $\gamma$ in $Q$, and the matrices $B_{a_\ell}, \dots, B_{a_1}$ appear as components in the summand $\bigoplus_{a\in S_{1,1}} \Mat(v_{\head(a)}\times v_{\tail(a)})$ from \eqref{eqn:rep4sums} on which $\varphi$ is the identity map. Thus, the map $\varphi^*$ identifies the generator $\tr_{\gamma'}$ of $\kk[\Rep(Q',v')]^{\GL(v')}$ with the function $\tr_\gamma\in \kk[\Rep(Q,v)]^{\GL(v)}$ for the corresponding cycle $\gamma$ in $Q$. These cycles $\gamma$ in $Q$ are precisely those that traverse only arrows in $Q$ that have both head and tail in $K$, so we obtain the first set of generators listed in Theorem~\ref{thm:Lusztig}.
  
 \medskip

 \noindent \textsc{Case 2:} \emph{the cycle $\gamma'$ in $Q'$ traverses at least one arrow with head or tail at vertex $\infty$}.

 \smallskip
 Since trace is invariant under cyclic permutation of matrices, we may assume without loss of generality that the tail of the cycle $\gamma'$ in $Q'$ is at vertex $\infty$. Also, since $v'_\infty=1$, it is enough to consider such cycles $\gamma'$ that don't touch $\infty$ except at the head and tail of the cycle. Moreover, 
 \begin{equation}
     \label{eqn:pdecomp}
\gamma'=\beta_c^{i}a_{\ell-1}\cdots a_2\alpha_b^{j}
 \end{equation}
 for some $b\in S_{0,1}$ and $c\in S_{1,0}$, with $1\leq i\leq v_{\head(c)}$ and $1\leq j\leq v_{\tail(b)}$, where the head and tail of each arrow $a_2, \dots, a_{\ell-1}$ in $Q'$ lies in $K$. Since $v'_\infty=1$, the trace function $\tr_{\gamma'}$  sends each point $B'=(B'_a)\in \Rep(Q',v')$ to the $1\times 1$ matrix 
 \[
 \tr_{\gamma'}(B') = B_{\beta_c^{i}}B_{a_{\ell-1}}\cdots B_{a_2}B_{\alpha_b^{j}}
 \]
  associated to the composition  of arrows appearing in \eqref{eqn:pdecomp}. More generally, fixing $b\in S_{0,1}$ and $c\in S_{1,0}$, but letting $1\leq i\leq v_{\head(c)}$ and $1\leq j\leq v_{\tail(b)}$ vary, produces a collection of $v_{\head(c)} v_{\tail(b)}$ scalars of the form $\tr_{\gamma'(i,j)}$, where now $\gamma'(i,j)$ denotes the cycle as in \eqref{eqn:pdecomp} for the various $1\leq i\leq v_{\head(c)}$ and $1\leq j\leq v_{\tail(b)}$. Examining the proof of Lemma~\ref{lem:varphi} shows that these scalars are precisely the entries in the $v_{\head(c)}\times v_{\tail(b)}$ matrix $B_cB_{a_{\ell-1}}\cdots B_{a_2}B_b$ for a point $(B_a)\in \Rep(Q,v)$. Therefore 
 \begin{equation}
     \label{eqn:case2} 
 \varphi^*(\tr_{\gamma'(i,j)}) = \pij
 \end{equation}
 for the path $p=ca_{\ell-1}\cdots a_2b$ in $Q$ and for $1\leq i\leq v_{\head(p)}$ and $1\leq j\leq v_{\tail(p)}$, where now we regard each of $a_2, \dots, a_{\ell-1}$ as in arrow in $Q$. Note that every such path $p$ in $Q$ has both head and tail in $K^c$, so every such $\pij$ is among the second set of generators listed in Theorem~\ref{thm:Lusztig}.

 \medskip
 
 As a result, each $\kk$-algebra generator of $\kk[\Rep(Q,v)]^{G_K}$ obtained via the isomorphism $\varphi^*$ either lies in the set \eqref{eqn:easygens}, or it is obtained as in \textsc{Case 1} or \textsc{Case 2}, so each is as in Theorem~\ref{thm:Lusztig}.
\end{proof}


\section{Lusztig's generators for a quiver algebra} 
\label{sec:quiveralgebras}
A \emph{relation} for a quiver $Q$ is a finite linear combination of paths in $Q$, each with the same head and tail. Thus, for each relation $g$, there are vertices $\tail(g), \head(g)\in Q_0$ such that $g$ lies in the $\kk$-vector subspace $e_{\head(g)} \kk Q e_{\tail(g)}$ spanned by paths with tail $\tail(g)$ and head $\head(g)$; explicitly, 
\begin{equation}
    \label{eqn:relation}
 g = \sum_{1\leq l\leq \ell} \lambda_l p_l,
\end{equation}
 where $\lambda_1, \dots,\lambda_\ell\in \kk$, and paths $p_1, \dots, p_\ell$ satisfy $\tail(p_l)=\tail(g)$ and $\head(p_l)=\head(g)$ for all $1\leq l\leq \ell$. 
 
 A \emph{quiver algebra} is a $\kk$-algebra $A$ for which there exists a quiver $Q$ and a surjective homomorphism
  \begin{equation}
   \label{eqn:quiveralgebra}
 \beta\colon \kk Q \longrightarrow A
 \end{equation}
 of $\kk$-algebras whose kernel is a two-sided ideal $\langle g_1, \dots, g_m\rangle$, where each $g_i$ is a relation in $Q$. The algebra $A$ is $Q_0$-bigraded, in the sense that $A=\bigoplus_{\tail, \head\in Q_0} e_{\head} A e_{\tail}$, where $e_{\head} A e_{\tail}$ is the $\kk$-vector subspace of $A$ spanned by the classes of paths with head and tail at vertex $\head$ and $\tail$ respectively.
 
 Let $v\in \NN^{Q_0}$ be a dimension vector. Then $B=(B_a)\in \Rep(Q,v)$ 
 \emph{satisfies the relation \eqref{eqn:relation}} if
 \begin{equation}
     \label{eqn:matrixrelation}
 \sum_{1\leq l \leq \ell}\lambda_l B_{p_l}= 0.
  \end{equation}
 The matrix equation  \eqref{eqn:matrixrelation} encodes a collection of $v_{\tail(g)}v_{\head(g)}$ equations in the entries of the matrices in $B$. Thus, the representation $B$ satisfies the relation \eqref{eqn:relation} if and only if $B$, when regarded as a point in the affine space $\Rep(Q,v)$, lies in the subscheme cut out by the equations
 \begin{equation}
     \label{eqn:variablerelation}
 \xpij{g}{ij}=\sum_{1\leq l \leq \ell} \lambda_l \plij =0
  \end{equation}
 for all $1\leq i\leq v_{\head(g)}$ and $1\leq j\leq v_{\tail(g)}$. More generally, for any quiver algebra $A =\kk Q/\langle g_1, \dots, g_m\rangle$, the \emph{representation scheme} $\Rep(A,v)$ is the subscheme of $\Rep(Q,v)$ cut out by the ideal
 \begin{equation}
 \label{eqn:idealI}
 I:=\big\langle \xpij{g_k}{ij} \mid 1\leq k\leq m, \; 1\leq i\leq v_{\head(g_k)}, \; 1\leq j\leq v_{\tail(g_k)}\big\rangle.
 \end{equation}
 This ideal does not depend on the choice of generators of the defining ideal $\langle g_1, \dots, g_m\rangle$ of $A$. Write
   \[
   \tau\colon \kk[\Rep(Q,v)]\longrightarrow \kk[\Rep(A,v)]
   \]
   for the surjective $\kk$-algebra homomorphism given by restricting functions.
 
 \begin{lemma}
 \label{lem:restrictfunctions}
 For a quiver algebra $A$ with presentation \eqref{eqn:quiveralgebra},
 fix vertices $\head, \tail\in Q_0$ and $f\in e_{\head} A e_{\tail}$. For $1\leq i\leq v_{\head}$, $1\leq j\leq v_{\tail}$, the function $\xpij{f}{ij}:= \xpij{p}{ij}\vert_{\Rep(A,v)}$ in $\kk[\Rep(A,v)]$ obtained by restriction is well-defined independent of the choice of lift $p\in \beta^{-1}(f)$. Thus, for any $p\in e_{\head}\kk Q e_{\tail}$, we have 
  \[
  \tau(\xpij{p}{ij})=\xpij{p}{ij}\vert_{\Rep(A,v)} = \xpij{\beta(p)}{ij}.
  \]
 Similarly, the trace function of each cycle $\gamma$ in $Q$ satisfies $\tau(\tr_\gamma)=\tr_\gamma\vert_{\Rep(A,v)} = \tr_{\beta(\gamma)}$. 
 \end{lemma}
 \begin{proof}
 Let $p,q\in \beta^{-1}(f)$, so $p-q\in \ker(\beta)=\langle g_1, \dots, g_m\rangle$. Write $p-q=\sum_k r_kg_ks_k$ for elements $r_k\in e_{\head}\kk Q e_{\head(g_k)}$ and $s_k\in e_{\tail(g_k)}\kk Qe_{\tail}$ for $1\leq k\leq m$. For any $1\leq i\leq v_{\head}$ and $1\leq j\leq v_{\tail}$, we have
  \[
   \xpij{p}{ij} - \xpij{q}{ij} 
   =
   \xpij{(p-q)}{ij} 
   = \sum_k \xpij{r_k g_k s_k}{ij} = \sum_k \sum_{k_1, k_2}\xpij{r_k}{ik_1}\xpij{g_k}{k_1k_2}\xpij{s_k}{k_2j}
   \]
 in $\kk[\Rep(Q^*,v)]$, where the third equality follows from Lemma~\ref{lem:product}.  The right-hand expression lies in the ideal \eqref{eqn:idealI} defining the scheme $\Rep(A,v)$, so the function $\tau(\xpij{p}{ij})=\xpij{p}{ij}\vert_{\Rep(A,v)}$ coincides with the function $\tau(\xpij{q}{ij})=\xpij{q}{ij}\vert_{\Rep(A,v)}$ in $\kk[\Rep(A,v)]$ as required. The second statement follows because $\tau(\xpij{p}{ij})=\xpij{p}{ij}\vert_{\Rep(A,v)}$. The final statement follows, because any cycle $\gamma$ passing through vertex $i\in Q_0$ satisfies $\tr_\gamma = \sum_{1\leq i\leq v_i} \xpij{\gamma}{ii}$. 
  \end{proof}

     Given $\kk[\Rep(A,v)]$ and the surjective $\kk$-algebra homomorphism $\beta\colon \kk Q\to A$, it is now straightforward to give generators of the $G_K$-invariant subalgebra $\kk[\Rep(A,v)]^{G_K}$.
     
 \begin{proposition}
 \label{prop:RepAinvariants}
  For any subset $K\subseteq Q_0$, the $\kk$-algebra $\kk[\Rep(A,v)]^{G_K}$ is generated by:
 \begin{enumerate}
 \item[\aLus] the trace functions $\tr_{\beta(\gamma)}$ associated to cycles $\gamma$ in $Q$ that traverse only arrows in $Q$ with head and tail in $K$; and
 \item[\bLus] the contraction functions $\xpij{\beta(p)}{ij}$ for paths $p$ in $Q$ with head and tail in $K^c$, where the indices satisfy $1\leq i\leq v_{\head(p)}$ and $1\leq j\leq v_{\tail(p)}$.
 \end{enumerate}
 \end{proposition}
 \begin{proof}
    Since $\GL(v)$ is reductive and $\kk$ has characteristic zero, the restriction of the surjective map $\tau$ to the $G_K$-invariant subalgebras is a surjective $\kk$-algebra homomorphism
     \begin{equation}
     \label{eqn:qK}
    \tau_K\colon \kk[\Rep(Q,v)]^{G_K}\longrightarrow \kk[\Rep(A,v)]^{G_K}.
     \end{equation}
     Trace and contraction functions generate $\kk[\Rep(Q,v)]^{G_K}$ by Theorem~\ref{thm:Lusztig}, so the given trace and contraction functions of classes in $A$ generate $\kk[\Rep(A,v)]^{G_K}$ by Lemma~\ref{lem:restrictfunctions} as required.
 \end{proof}

\begin{remark}
\label{rem:finitelymany}
 Only finitely many of the generators listed in Proposition~\ref{prop:RepAinvariants} are required because the algebra $\kk[\Rep(A,v)]$ is finitely generated over $\kk$ and $G_K$ is reductive.
\end{remark}

Before working towards an explicit description of the kernel of the map from \eqref{eqn:qK}, we introduce a general construction. For any arrow $a\in Q_1$, consider the ideal 
\[
I_a:=\big\langle \xaij{a}{ij} \mid 1\leq i\leq v_{\head(a)} \text{ and }1\leq j\leq v_{\tail(a)}\big\rangle
\]
in the polynomial ring $\kk[\Rep(Q,v)]$.

\begin{lemma}
\label{lem:traversesa}
\begin{enumerate}
   \item[\one] A nontrivial path $p$ in $Q$ traverses arrow $a$ if and only if $\xpij{p}{ij}\in I_a$ for all $1\leq i\leq v_{\head(a)} \text{ and }1\leq j\leq v_{\tail(a)}$.
    \item[\two] A nontrivial cycle $\gamma$ traverses arrow $a$ if and only if $\tr_\gamma\in I_a$.
\end{enumerate}
\end{lemma}
\begin{proof}
   For \one, write $p=a_\ell\cdots a_1$ for $a_1, \dots, a_\ell\in Q_1$. One direction is immediate from Lemma~\ref{lem:product}. For the converse, suppose $\xpij{p}{ij}\in I_a$ for all $1\leq i\leq v_{\head(a)} \text{ and }1\leq j\leq v_{\tail(a)}$. Lemma~\ref{lem:product} gives
   \[
    \sum_{k_1, \dots, k_{\ell-1}} \xpij{a_{\ell}}{ik_{\ell-1}}\xpij{a_{\ell-1}}{k_{\ell-1}k_{\ell-2}}\cdots \xpij{a_{1}}{k_{1}j} = \xpij{p}{ij} = \sum_{k,l} h_{k,l}\xaij{a}{kl}
   \]
   for some $h_{k,l}\in \kk[\Rep(Q,v)]$. Compare nonzero terms to see that $a=a_k$ for some $1\leq k\leq \ell$. Part \two\ is similar: in each direction, combine Lemma~\ref{lem:product} with the fact that $\tr_\gamma =\sum_{1\leq i\leq v_{\tail(\gamma)}} \xpij{\gamma}{ii}$. 
\end{proof}

\begin{proposition}
\label{prop:IainHKinvariants}
Let $K\subseteq Q_0$. The ideal $I_{a}\cap \kk[\Rep(Q,v)]^{G_K}$ in $\kk[\Rep(Q,v)]^{G_K}$ is generated by:
\begin{enumerate}
    \item[\aLus] the trace functions $\tr_\gamma$ associated to cycles $\gamma$ in $Q$ that traverse only arrows with head and tail in $K$, including arrow $a$; and
    \item[\bLus] the contraction functions $\xpij{p}{ij}$ associated to paths $p$ in $Q$ that traverse arrows including arrow $a$, where both the head and tail of $p$ lie in $K^c$, and where $1\leq i\leq v_{\head(p)}$ and $1\leq j\leq v_{\tail(p)}$. 
\end{enumerate}
\end{proposition}

\begin{remark}
 If either $\head(a)$ or $\tail(a)$ does not lie in $K$, then there are no cycles $\gamma$ of type \aLus.    
\end{remark}

\begin{proof}
For a path $p$ in $Q$ traversing $a$ with $\tail(p), \head(p)\not\in K$, and for $1\leq i\leq v_{\head(p)}$ and  $1\leq j\leq v_{\tail(p)}$, the previous lemma gives $\xpij{p}{ij}\in I_{a}$. Since  $\tail(p), \head(p)\not\in K$, the function $\xpij{p}{ij}$ is $G_K$-invariant, so $\xpij{p}{ij}\in I_{a}\cap \kk[\Rep(Q,v)]^{G_K}$. Similarly, for a cycle $\gamma$ traversing $a$ with $\tail(a)\in K$, the previous lemma gives $\tr_\gamma\in I_{a}$. Also,  $\tr_\gamma$ is $G_K$-invariant, so $\tr_\gamma\in I_{a}\cap \kk[\Rep(Q,v)]^{G_K}$. Thus, the ideal in $\kk[\Rep(Q,v)]^{G_K}$ generated by all functions of type \aLus\ and \bLus\ is contained in the ideal $I_{a}\cap \kk[\Rep(Q,v)]^{G_K}$.

For the opposite inclusion, each element of $I_{a}$ is of the form 
\begin{equation}
    \label{eqn:elementinIa}
f=\sum_{1\leq i\leq v_{\head(a)}} \sum_{1\leq j\leq v_{\tail(a)}} h_{i,j}\xpij{a}{ij}
\end{equation}
 for some $h_{i,j}\in \kk[\Rep(Q,v)]$. Proposition~\ref{prop:RepAinvariants} shows that every element of the algebra $\kk[\Rep(Q,v)]^{G_K}$ can be written as a linear combination of monomials of the form  
 \begin{equation}
     \label{eqn:elementinkRepQvHK}
 \prod_{p, \gamma, i, j} \xpij{p}{ij}^{m_{p,ij}} \tr_\gamma^{n_\gamma},
  \end{equation}
 for some $m_{p,ij}, n_\gamma\in \mathbb{N}$, where the sum is taken over the appropriate paths $p$, cycles $\gamma$, and indices $i, j$ as listed in Proposition~\ref{prop:RepAinvariants}. Thus, for $f\in I_a\cap \kk[\Rep(Q,v)]^{G_K}$, every term 
of $f$ is divisible by some $\xpij{a}{ij}$ by \eqref{eqn:elementinIa}, and moreover, is also a scalar multiple of a monomial of the form \eqref{eqn:elementinkRepQvHK}. Thus, in every such term, a path $p$ or cycle $\gamma$ appearing in the sum traverses arrow $a$, and it satisfies the other conditions listed in \aLus\ or \bLus\ by Proposition~\ref{prop:RepAinvariants}. Therefore every nonzero term of $f$ lies in the ideal of $\kk[\Rep(Q,v)]^{G_K}$ generated by the functions of type \aLus\ or \bLus\  as required. 
\end{proof}

\section{Relations between Lusztig's generators}
\label{sec:relations}
Consider again a quiver algebra $A$ with surjective $\kk$-algebra homomorphism $\beta\colon 
\kk Q\to A$ satisfying $\ker(\beta)=\langle g_1, \dots, g_m\rangle$, where for $1\leq k\leq m$, there exist vertices  $\tail_k, \head_k\in Q_0$ such that $g_k\in e_{\head_k}\kk Qe_{\tail_k}$. Each relation $g_k$ is of the form \eqref{eqn:relation} and hence $\xpij{g_k}{ij}\in \kk[\Rep(Q,v)]$ for $1\leq i\leq v_{\head_k}$ and $1\leq j\leq v_{\tail_k}$. 

 Augment the quiver $Q$ by introducing additional arrows $a_1, \dots, a_m$, where $\tail(a_k)=\tail_k$ and $\head(a_k)=\head_k$ for each $1\leq k\leq m$. The resulting quiver $\overline{Q}$ has the same vertex set $Q_0$ as $Q$, but the arrow set is $\overline{Q}_1:= Q_1\cup \{a_1, \dots, a_m\}$. Notice that 
\begin{equation}
    \label{eqn:kRepQ'v}
\kk[\Rep(\overline{Q},v)]\cong \kk[\xpij{a_k}{ij} \mid  1\leq k\leq m, 1\leq  i\leq v_{\head_k}, 1\leq j\leq v_{\tail_k}] \otimes_\kk \kk[\Rep(Q,v)],
\end{equation}
 and moreover, the action of $G_K$ on $\kk[\Rep(\overline{Q},v)]$ is compatible with the $G_K$-action on $\kk[\Rep(Q,v)]$. Consider the ideals $I_{a_1}, \dots, I_{a_m}$ associated to the arrows $a_1, \dots, a_m$, and write 
\[
\overline{I}:= I_{a_1}+\cdots + I_{a_m} = \big\langle \xaij{a_k}{ij} \mid 1\leq k\leq m, 1\leq i\leq v_{\head_k} \text{ and }1\leq j\leq v_{\tail_k}\big\rangle
\]
for the ideal sum in $\kk[\Rep(\overline{Q},v)]$. 

\begin{corollary}
\label{cor:sumideal}
Let $K\subseteq Q_0$. The ideal $\overline{I}\cap \kk[\Rep(\overline{Q},v)]^{G_K}$ in $\kk[\Rep(\overline{Q},v)]^{G_K}$ is generated by:
\begin{enumerate}
    \item[\aLus] the functions $\tr_\gamma$ associated to cycles $\gamma$ in $\overline{Q}$ that traverse only arrows with head and tail in $K$, including an arrow $a_k$ for some $1\leq k\leq m$; and
    \item[\bLus] the functions $\xpij{p}{ij}$ associated to paths $p$ in $\overline{Q}$ that traverse arrows including $a_k$ for some $1\leq k\leq m$, where both the head and tail of $p$ lie in $K^c$, where $1\leq i\leq v_{\head(p)}$, and $1\leq j\leq v_{\tail(p)}$. 
\end{enumerate}
\end{corollary}
 \begin{proof}
  Each ideal $I_{a_k}$ in $\kk[\Rep(\overline{Q},v)]$ cuts out an $G_K$-invariant subscheme of $\Rep(\overline{Q},v)$, so $I_{a_k}$ is $G_K$-invariant. Apply the general result from Lemma~\ref{lem:Reynolds} below to see that 
  \[
  \overline{I}\cap \kk[\Rep(\overline{Q},v)]^{G_K} = \sum_{1\leq k\leq m} \Big(I_{a_k}\cap \kk[\Rep(\overline{Q},v)]^{G_K}\Big),
  \]
 so $\overline{I}\cap \kk[\Rep(\overline{Q},v)]^{G_K}$ is generated by the union of generators of the ideals $I_{a_k}\cap \kk[\Rep(\overline{Q},v)]^{G_K}$ for $1\leq k\leq m$. The result follows by applying Proposition~\ref{prop:IainHKinvariants} to each of $a_1, \dots, a_m$.
 \end{proof}

\begin{lemma}\label{lem:Reynolds}
  Let $R$ be a commutative $\kk$-algebra and $G$ a reductive group acting on $R$ as a rational representation. Then, for any collection of $G$-invariant ideals $\{I_i\}_i$ of $R$, we have
    $$\Big(\sum_i I_i\Big)\cap R^G=\sum_i\big(I_i\cap R^G\big).$$
\end{lemma}
\begin{proof}
The result follow by considering the restriction of the Reynolds operator $E\colon R\to R^{G}$ to each ideal $I_i$; see \cite[Corollary~1.4]{BertramGIT} for details.
\end{proof}

 Recall from \eqref{eqn:idealI} that $\kk[\Rep(A,v)]$ is isomorphic to the quotient of $\kk[\Rep(Q,v)]$ by the ideal 
  \[
 I = \big\langle  \xpij{g_k}{ij}\mid 1\leq k\leq m, 1\leq i\leq \head_k, 1\leq j\leq \tail_k \big\rangle,
 \]
 and moreover, restriction to the $G_K$-invariant subalgebra determines the surjective map 
  \[
 \tau_K\colon \kk[\Rep(Q,v)]^{G_K} \longrightarrow \kk[\Rep(A,v)]^{G_K}
 \]
 from \eqref{eqn:qK}.
 
 \begin{theorem}
 \label{thm:I_K}
 Let $K\subseteq Q_0$. The kernel of $\tau_K$ is the ideal in $\kk[\Rep(Q,v)]^{G_K}$ generated by:
 \begin{enumerate}
    \item[\aLus] the functions $\tr_{\gamma}$, where $\gamma\in e_k\ker(\beta) e_k$ for some $k\in K$; and
  \item[\bLus] the functions $\xpij{p}{ij}$, where $p\in e_{\head}\ker(\beta) e_{\tail}$ for $\head, \tail\in K^c$, where $1\leq i\leq v_{\head}$ and $1\leq j\leq v_{\tail}$.
\end{enumerate}
 \end{theorem}
 \begin{proof}
Define a surjective map of polynomial rings $f'\colon \kk[\Rep(\overline{Q},v)]\to \kk[\Rep(Q,v)]$ by setting
 \[
 f'(\xpij{a}{ij}) = \left\{\begin{array}{lr}
 \xpij{a}{ij} &  \text{ if }a\in Q_1; \\
 \xpij{g_k}{ij} & \quad \text{if }a=a_k\text{ for some }1\leq k\leq m, \end{array}\right.
 \]
 where $a\in Q_1'$, $1\leq i\leq v_{\head(a)}$ and $1\leq j\leq v_{\tail(a)}$. Observe that $f'$ is $G_K$-equivariant. The description of $\kk[\Rep(\overline{Q},v)]$ from \eqref{eqn:kRepQ'v} gives
 \[
 \ker(f') = \big\langle \xpij{a_k}{ij}-\xpij{g_k}{ij} \mid 1\leq k\leq m, 1\leq  i\leq v_{\head_k}\text{ and }1\leq j\leq v_{\tail_k}\big\rangle. 
 \]
 The map $f'$ and the map $\tau$ from the proof of Proposition~\ref{prop:RepAinvariants} fit into a commutative diagram 
     \begin{equation}
 \label{eqn:diagram3.6}
\begin{tikzcd}
  \kk[\Rep(\overline{Q},v)]\ar[r,"f'"] & \kk[\Rep(Q,v)] \ar[r,"\tau"] & \kk[\Rep(A,v)] \\
\kk[\Rep(\overline{Q},v)]^{G_K}\ar[r,"f'_K"]\ar[u,"\iota'"]&\kk[\Rep(Q,v)]^{G_K}\ar[r,"\tau_K"]\ar[u,"\iota"] & \kk[\Rep(A,v)]^{G_K}\ar[u]
  \end{tikzcd}
  \end{equation}
 where the vertical maps are the natural inclusions and the horizontal maps are all surjective; note that $f'_K$ and $\tau_K$ are obtained from $f'$ and $\tau$ respectively by restriction.  Notice that 
  \[
 \ker(\tau)=I = \big\langle \xpij{g_k}{ij}\mid 1\leq k\leq m, 1\leq i\leq \head_k, 1\leq j\leq \tail_k\big\rangle.
 \]
  The ideal of interest is $\ker(\tau_K)=\iota^{-1}(I) = I\cap \kk[\Rep(Q,v)]^{G_K}$. If we contract $I$ along $f'$ we obtain the ideal $(f')^{-1}(I) = \overline{I}+\ker(f')$. Commutativity of the diagram implies that 
 \begin{eqnarray}
 (f'_K)^{-1}(\ker(\tau_K)) & = &(\iota\circ f_K')^{-1}(I) \nonumber \\
 & = &  (f'\circ \iota')^{-1}(I) \nonumber \\
 & = & (\iota')^{-1}\big(\overline{I}+\ker(f')\big) \nonumber \\
 & = &  (\overline{I}+\ker(f'))\cap \kk[\Rep(\overline{Q},v)]^{G_K} \nonumber \\
 & = &  \big(\overline{I}\cap \kk[\Rep(\overline{Q},v)]^{G_K}\big)+\big(\ker(f')\cap \kk[\Rep(\overline{Q},v)]^{G_K}\big), \label{eqn:sumideal}
 \end{eqnarray}
 where the last equality holds by Lemma \ref{lem:Reynolds}; here, $\ker(f')$ is $G_K$-invariant as $f'$ is $G_K$-equivariant.
 
 Since $f_K'$ is surjective, the extension along $f'_K$ of the ideal $ (f'_K)^{-1}(\ker(\tau_K))$  equals $\ker(\tau_K)$, so $\ker(\tau_K)$ is generated by the images under $f'_K$ of the generators of the ideal in \eqref{eqn:sumideal}; this coincides with the images of the generators of $\overline{I}\cap \kk[\Rep(\overline{Q},v)]^{G_K}$. Those generators are given in Corollary~\ref{cor:sumideal}, so we need only apply $f'$ to those generators to obtain generators of $\ker(\tau_K)$. If the cycle $\gamma$ or the path $p$ appearing in one of the generators from Corollary~\ref{cor:sumideal} traverses precisely one arrow $a_k$ for some $1\leq k\leq m$, then the image under $f'$ of this generator is of the form $\tr_{q g_k q'}$ or $\xpij{pg_kp'}{ij}$ as required. More generally, a path or cycle might traverse more than one such arrow, in which case its image under $f'$ is some linear combination of functions of the form $\tr_{q g_k q'}$ or $\xpij{pg_kp'}{ij}$. Therefore, the given trace and contraction functions provide generators for the ideal $\ker(\tau_K)$. 
 \end{proof}

 \begin{example}
    \label{sec:example}
 Let $\Gamma\subset \SL(2,\mathbb{C})$ be the cyclic group of order two whose nontrivial element acts as minus the identity. The McKay quiver of $\Gamma$ is the quiver $Q$ with vertex set $\{0,1\}$ and arrow set $\{c, d, e, f\}$, where the arrows are $c, d\colon 0\to 1$ and $e, f\colon 1\to 0$.  The relations $g_1= fc-ed$ and $g_2 = de-cf$ define the preprojective algebra $\Pi:= \kk Q/\langle g_1, g_2\rangle$ of type $A_1$. 
 
 For the dimension vector $v=(2,2)$, we compute the appropriate matrix products to obtain all eight polynomials $\xpij{g_k}{ij}$ for $1\leq i,j,k\leq 2$ as in \eqref{eqn:variablerelation} that generate the ideal   
  \[
  \ker(\tau)=\left(\begin{array}{c} 
  \scriptstyle{\xpij{c}{11}  \xpij{f}{11}+ \xpij{c}{21}  \xpij{f}{12} - \xpij{d}{11}  \xpij{e}{11}- \xpij{d}{21}  \xpij{e}{12},\;\; 
  \xpij{c}{12}  \xpij{f}{11}+ \xpij{c}{22}  \xpij{f}{12}  - \xpij{d}{12}  \xpij{e}{11}- \xpij{d}{22}  \xpij{e}{12}} \\
  \scriptstyle{\xpij{c}{11}  \xpij{f}{21}+ \xpij{c}{21}  \xpij{f}{22} - \xpij{d}{11}  \xpij{e}{21}- \xpij{d}{21}  \xpij{e}{22},\;\; 
  \xpij{c}{12}  \xpij{f}{21}+ \xpij{c}{22}  \xpij{f}{22} - \xpij{d}{12}  \xpij{e}{21}- \xpij{d}{22}  \xpij{e}{22}} \\
   \scriptstyle{  \xpij{d}{11}  \xpij{e}{11}+ \xpij{d}{12}  \xpij{e}{21}- \xpij{c}{11}  \xpij{f}{11}- \xpij{c}{12}  \xpij{f}{21},\;\;
     \xpij{d}{11}  \xpij{e}{12}+ \xpij{d}{12}  \xpij{e}{22}- \xpij{c}{11}  \xpij{f}{12}- \xpij{c}{12}  \xpij{f}{22}} \\
    \scriptstyle{ \xpij{d}{21}  \xpij{e}{11}+ \xpij{d}{22}  \xpij{e}{21}- \xpij{c}{21}  \xpij{f}{11}- \xpij{c}{22}  \xpij{f}{21},\;\;
     \xpij{d}{21}  \xpij{e}{12}+ \xpij{d}{22}  \xpij{e}{22}- \xpij{c}{21}  \xpij{f}{12}- \xpij{c}{22}  \xpij{f}{22}}
    \end{array}\right)
    \]
   satisfying 
   \[
   \kk[\Rep(\Pi,v)]\cong \kk[\xpij{c}{ij}, \xpij{d}{ij}, \xpij{e}{ij}, \xpij{f}{ij} \mid 1\leq i,j\leq 2]/\ker(\tau).
   \]   
Set $K=\{1\}$.  Each nontrivial path $p$ in $Q$ with head and tail in $K^c=\{0\}$ is the concatenation of finitely many copies of each of the four cycles $ec, fc, fd, ed$ in $Q$ with tail at 0, taken in some order. The relation $g_1$ allows us to ignore the last of these cycles when working in $\Pi$. As a result, the $2\times 2$ matrices  $B_{ec}, B_{fc}, B_{fd}$ associated to the first three cycles determine the twelve generators 
    \[
    \left\{\begin{array}{cccc} 
   \!\!  \xpij{c}{11}  \xpij{e}{11}+ \xpij{c}{21}  \xpij{e}{12},\!\! & 
     \xpij{c}{12}  \xpij{e}{11}+ \xpij{c}{22}  \xpij{e}{12},\!\! &
     \xpij{c}{11}  \xpij{e}{21}+ \xpij{c}{21}  \xpij{e}{22},\!\! &
     \xpij{c}{12}  \xpij{e}{21}+ \xpij{c}{22}  \xpij{e}{22}\!\! \\
    \!\! \xpij{c}{11}  \xpij{f}{11}+ \xpij{c}{21}  \xpij{f}{12},\!\! &
     \xpij{c}{12}  \xpij{f}{11}+ \xpij{c}{22}  \xpij{f}{12},\!\! & 
     \xpij{c}{11}  \xpij{f}{21}+ \xpij{c}{21}  \xpij{f}{22},\!\! & 
     \xpij{c}{12}  \xpij{f}{21}+ \xpij{c}{22}  \xpij{f}{22}\!\! \\
    \!\! \xpij{d}{11}  \xpij{f}{11}+ \xpij{d}{21}  \xpij{f}{12},\!\! & 
     \xpij{d}{12}  \xpij{f}{11}+ \xpij{d}{22}  \xpij{f}{12},\!\! & 
     \xpij{d}{11}  \xpij{f}{21}+ \xpij{d}{21}  \xpij{f}{22},\!\! & 
     \xpij{d}{12}  \xpij{f}{21}+ \xpij{d}{22}  \xpij{f}{22}\!\!
    \end{array}\right\}\]
     of $\kk[\Rep(\Pi,v)]^{G_K}$. Proposition~\ref{prop:RepAinvariants} states that one should also include the trace functions of cycles based at $1\in K$, but these are redundant because trace is cyclic, e.g.\ $\tr_{ce} = \tr_{ec}=\sum_{1\leq i\leq 2} \xpij{ec}{ii}$.
     
          To compute $\ker(\tau_K)= \ker(\tau)\cap  \kk[\Rep(Q,v)]^{G_K}$, introduce variables $x_{ij}, y_{ij}, z_{ij}$ for $1\leq i,j\leq 2$, one for each generator of $\kk[\Rep(\Pi,v)]^{G_K}$ listed above. We use Macaulay2~\cite{M2} to eliminate the sixteen variables $\xpij{c}{ij}, \xpij{d}{ij}, \xpij{e}{ij}, \xpij{f}{ij}$ from the ideal 
    \begin{equation}
        \label{eqn:elimination}
    \left(\begin{array}{cc} 
    x_{11}- (\xpij{c}{11}  \xpij{e}{11}+ \xpij{c}{21}  \xpij{e}{12}), \;\;
    x_{12}- (\xpij{c}{12}  \xpij{e}{11}+ \xpij{c}{22}  \xpij{e}{12}) \\
    x_{21}- (\xpij{c}{11}  \xpij{e}{21}+ \xpij{c}{21}  \xpij{e}{22}), \;\;
    x_{22}- (\xpij{c}{12}  \xpij{e}{21}+ \xpij{c}{22}  \xpij{e}{22}) \\
    y_{11}- (\xpij{c}{11}  \xpij{f}{11}+ \xpij{c}{21}  \xpij{f}{12}), \;\;
    y_{12}- (\xpij{c}{12}  \xpij{f}{11}+ \xpij{c}{22}  \xpij{f}{12}) \\
    y_{21}- (\xpij{c}{11}  \xpij{f}{21}+ \xpij{c}{21}  \xpij{f}{22}), \;\;
    y_{22}- (\xpij{c}{12}  \xpij{f}{21}+ \xpij{c}{22}  \xpij{f}{22}) \\
    z_{11}- (\xpij{d}{11}  \xpij{f}{11}+ \xpij{d}{21}  \xpij{f}{12}), \;\;
    z_{12}- (\xpij{d}{12}  \xpij{f}{11}+ \xpij{d}{22}  \xpij{f}{12}) \\
    z_{21}- (\xpij{d}{11}  \xpij{f}{21}+ \xpij{d}{21}  \xpij{f}{22}), \;\;
    z_{22}- (\xpij{d}{12}  \xpij{f}{21}+ \xpij{d}{22}  \xpij{f}{22})
    \end{array}\right).
    \end{equation}
The resulting elimination ideal in the polynomial ring 
$\kk[x_{ij}, y_{ij}, z_{ij}\mid 1\leq i,j\leq 2]$ is equal to 
\begin{equation}
\label{eqn:eliminationideal}
\left(\begin{array}{c} 
 x_{12}  y_{21} - x_{21}  y_{12},\;\; 
 x_{12}  z_{21} - x_{21}  z_{12},\;\;
   y_{12}  z_{21} - y_{21}  z_{12}, \\
 x_{11}  z_{11}+ x_{21}  z_{12}- y_{11}  y_{11}- y_{12}  y_{21},\;\; x_{12}  z_{11}+ x_{22}  z_{12}- y_{11}  y_{12}- y_{12}  y_{22} \\
 x_{11}  z_{21}+ x_{21}  z_{22} - y_{11}  y_{21}- y_{21}  y_{22}, \;\; x_{12}  z_{21}+ x_{22}  z_{22}- y_{12}  y_{21}- y_{22}  y_{22} \\
 x_{12}  y_{22}- x_{12}  y_{11}+ x_{11}  y_{12}- x_{22}  y_{12}, \;\;  x_{21}  y_{11}- x_{11}  y_{21}+ x_{22}  y_{21}- x_{21}  y_{22}\\
 x_{12}  z_{22} - x_{12}  z_{11}+ x_{11}  z_{12}- x_{22}  z_{12}, \;\; x_{21}  z_{11}- x_{11}  z_{21}+ x_{22}  z_{21}- x_{21}  z_{22}\\
 y_{12}  z_{11}- y_{11}  z_{12}+ y_{22}  z_{12}- y_{12}  z_{22}, \;\;  y_{21}  z_{22} - y_{21}  z_{11}+ y_{11}  z_{21}- y_{22}  z_{21}
      \end{array}\right)
    \end{equation}
One can check~\cite{M2} that this ideal is prime.
 
 Substitute the variables $\xpij{c}{ij}, \xpij{d}{ij}, \xpij{e}{ij}, \xpij{f}{ij}$ back in using \eqref{eqn:elimination} to rewrite this ideal via generators appearing in Theorem~\ref{thm:I_K}, namely 
\begin{equation}
    \label{eqn:kertauK}
\ker(\tau_K)=\left(\begin{array}{c}
\xpij{cg_2b}{ij},\;\;
\xpij{g_1cb}{ij}+\xpij{dag_1}{ij}- \xpij{cg_2b}{ij}-\xpij{dg_2a}{ij} \\ 
\xpij{g_1db}{ij}-\xpij{dg_2b}{ij}, \;\;
\xpij{cag_1}{ij}-\xpij{cg_2a}{ij}
\end{array}
\mid 1\leq i,j,\leq 2
\right);
\end{equation}
here, we present sixteen generators for $\ker(\tau_K)$, but some of these generators agree for $i=j$, leaving the thirteen generators from \eqref{eqn:eliminationideal}. Thus, the invariant subalgebra $\kk[\Rep(\Pi,v)]^{G_K}$ is isomorphic to the quotient of the polynomial ring $\kk[x_{ij}, y_{ij}, z_{ij}\mid 1\leq i,j\leq 2]$ by the ideal \eqref{eqn:eliminationideal}, or equivalently, the quotient of $\kk[\Rep(Q,v)]^{G_K}$ by the ideal $\ker(\tau_K)$ from \eqref{eqn:kertauK}. 
 \end{example}

\begin{remark}
    For a geometric interpretation of Example~\ref{sec:example}, the affine quotient of $\Rep(\Pi,v)\git \GL(v)$ is isomorphic to $\Sym^2(\mathbb{A}^2/\Gamma)$, while for $K=\{1\}$, the affine quotient $X:=\Rep(\Pi,v)\git G_K$ admits a residual action by $\GL(2)$. In this case, we show in \cite{CY25} that for the character $1\in \GL(2)^\vee$, the GIT quotient $X\git_1 \GL(2)$ is isomorphic to $\Hilb^2(\mathbb{A}^2/\Gamma)$, the Hilbert scheme of $2$-points on the Kleinian singularity. In particular, since the ideal \eqref{eqn:eliminationideal} is prime, the scheme $X=\Spec \kk[\Rep(\Pi,v)]^{G_K}$ is reduced and irreducible and hence $X\git_1 \GL(2)\cong \Hilb^2(\mathbb{A}^2/\Gamma)$ is also reduced and irreducible. 
    \end{remark}

 \small{
 
}

   \end{document}